\documentclass[12pt,a4paper]{amsart}
\usepackage{amsfonts,amsmath,amssymb}
\usepackage{hyperref}
\usepackage[all]{xy}

\newtheorem*{theorem*}{Theorem A}
\newtheorem*{theoremm*}{Theorem B}
\newtheorem*{theoremm1*}{Theorem A'}
\newtheorem*{theoremmm*}{Theorem C}
\newtheorem{lemma}{Lemma}[subsection]

\newtheorem{proposition}[lemma]{Proposition}
\newtheorem{remark}[lemma]{Remark}

\newtheorem{theorem}[lemma]{Theorem}
\newtheorem{definition}[lemma]{Definition}
\newtheorem{notation}[lemma]{Notation}
\newtheorem{property}[lemma]{Property}
\newtheorem{corollary}[lemma]{Corollary}

\baselineskip 18pt \textwidth 16cm  \theoremstyle{plain}

\newcommand{\oo}{\mathcal{O}}

\newcommand{\cc}{\mathbb{C}}

\newcommand{\rr}{\mathbb{R}}

\newcommand{\Fre}{{Fr\'{e}chet \,}}

\newcommand{\Z}{{\mathbb Z}}

\newcommand{\R}{{\mathbb R}}
\newcommand{\C}{{\mathbb C}}

\newcommand{\Sc}{{\mathcal S}}

\newcommand{\Fou}{{\mathcal{F}}}

\newcommand{\g}{{\mathfrak{g}}}
\newcommand{\h}{{\mathfrak{h}}}

\newcommand{\Supp}{\mathrm{Supp}}

\newcommand{\gd}{\g^{\sigma}}

\begin{document}

\author{Avraham Aizenbud}

\address{Avraham Aizenbud, Faculty of Mathematics
and Computer Science, The Weizmann Institute of Science POB 26,
Rehovot 76100, ISRAEL.} \email{aizenr@yahoo.com}

\author{Eitan Sayag}

\address{Eitan Sayag, Einstein Institute of Mathematics,
Edmond J. Safra Campus, Givat Ram, The Hebrew University of
Jerusalem, Jerusalem, 91904, Israel} \email{sayag@math.huji.ac.il}

\title[The Gelfand property of $(GL_{2n}(\rr),Sp_{2n}(\rr))$]{Invariant
distributions on non-distinguished nilpotent orbits with
application to the Gelfand property of
$(GL_{2n}(\rr),Sp_{2n}(\rr))$}
\date{\today}

%
%
%
%
%
%
%
%
%

\begin{abstract}
We study invariant distributions on the tangent space to a
symmetric space. We prove that an invariant distribution with the
property that both its support and the support of its Fourier
transform are contained in the set of non-distinguished nilpotent
orbits, must vanish. We deduce, using recent developments in the
theory of invariant distributions on symmetric spaces that the
symmetric pair $(GL_{2n}(\rr),Sp_{2n}(\rr))$ is a Gelfand pair.
More precisely, we show that for any irreducible smooth admissible
\Fre representation $(\pi,E)$ of $GL_{2n}(\rr)$ the space of
continuous functionals $Hom_{Sp_{2n}(\rr)}(E,\cc)$ is at most one
dimensional.
Such a result was previously proven for $p$-adic fields in
\cite{HR} for $\cc$ in \cite{S}.
\end{abstract}

 \maketitle
\tableofcontents
\section{Introduction}\label{intro}


Let $(V,\omega)$ be a symplectic vector space over $\rr$.
Consider the standard imbedding
$Sp(V,\omega) \subset GL(V)$ and the natural action of
$Sp(V,\omega) \times Sp(V,\omega)$ on $GL(V)$. In this paper we
prove the following theorem:

\begin{theorem*}
 Any $Sp(V,\omega) \times Sp(V,\omega)$ - invariant distribution on $GL(V)$
is invariant with respect to transposition.
\end{theorem*}

It has the following corollary in representation theory:

\begin{theoremm*}
 Let $(V,\omega)$ be a symplectic vector space and let E be an irreducible admissible smooth
\Fre representation of $GL(V).$ Then $$dimHom_{Sp(V)}(E,\cc) \leq
1$$
\end{theoremm*}

Theorem B is deduced from Theorem A using the Gelfand-Kazhdan
method (adapted to the archimedean case in [AGS]).




The analogue of Theorem A and Theorem B for non-archimedean fields
were proven in \cite{HR} using the method of Gelfand and Kazhdan.
A simple argument over finite fields is explained in \cite{GG} and
using this a simpler proof of the non-archimedean case was written
in \cite{OS}. Recently, one of us, using the ideas of \cite{AG_HC}
extended the result to the case $F=\cc$ (see \cite{S}).

Our proof of Theorem A is based on the methods of \cite{AG_HC}. In
that work the notion of regular symmetric pair was introduced and
shown to be a useful tool in the verification of the Gelfand
property . Thus, the main result of the present work is the {\it
regularity} of the symmetric pair $(GL(V),Sp(V,\omega).$ In
previous works the proof of regularity of symmetric pairs was
based either on some simple considerations or on a criterion that
requires negativity of certain eigenvalues (this was implicit in
\cite{JR}, \cite{RR} and was explicated in \cite{AG_HC},
\cite{AG_JR}, \cite{AG_regular}, \cite{S}).

The pair $(GL(V),Sp(V,\omega))$ does not satisfy the above
mentioned criterion and requires new techniques.

\subsection{Main ingredients of the proof}
$ $\\
To show regularity we study distributions on the space
$\g^{\sigma}=\{X \in gl_{2n}: JX=XJ\}$ where $J =
\begin{pmatrix}
  0_{n} & Id_{n} \\
  -Id_{n} & 0_{n}
\end{pmatrix}.$ More precisely, we are interested in those
distributions that are invariant with respect to the conjugation
action of $Sp_{2n}$ and supported on the nilpotent cone. To
classify the nilpotent orbits of the action we use the method of
\cite{GG} to identify these orbits with nilpotent orbits of the
adjoint action of $GL_{n}$ on its Lie algebra. This allows us to
show that there exists a unique {\it distinguished} nilpotent
orbit $\oo$ and that this orbit is open in the nilpotent cone.
Next, we use the theory of $D$ modules, as in \cite{AG_strong}, to
prove that there are no distributions supported on
non-distinguished orbits whose Fourier transform is also supported
on non-distinguished orbits (see Theorem \ref{maintrick}).

\subsection{Structure of the paper}
$ $\\
In section \ref{sec1} we give some preliminaries on distributions,
symmetric pairs and Gelfand pairs. We introduce the notion of
regular symmetric pairs and show that Theorem 7.4.5 of
\cite{AG_HC} and the results of \cite{S} allow us to reduce the
Gelfand property of the pair in question to proving that the pair
is regular. In section \ref{sec2} we prove the main technical
result on distributions,  Theorem \ref{maintrick}. It states that
under certain conditions there are no distributions supported on
non-distinguished nilpotent orbits. The proof is based on the
theory of $D$-modules. In section \ref{sec3} we use Theorem
\ref{maintrick} to prove that the pair $(GL(V),Sp(V,\omega))$ is
regular.


\subsection{Acknowledgements.}
$ $\\
We thank Dmitry Gourevitch and Omer Offen for fruitful
discussions. Part of the work on this paper was done while the
authors visited the Max Planck Institute for Mathematics in Bonn.
The visit of the first named author was funded by the Bonn
International Graduate School. The visit of the second named
author was partially funded by the Landau center of the Hebrew
University.

\section{Preliminaries}\label{sec1}
\subsection{Notations on invariant distributions}
\subsubsection{Schwartz distributions on Nash manifolds}
$ $\\
We will use the theory of Schwartz functions and distributions as
developed in \cite{AG_Sch}. This theory is developed for Nash
manifolds. Nash manifolds are smooth semi-algebraic manifolds but
in the present work only smooth real algebraic manifolds are
considered. Therefore the reader can safely replace the word {\it
Nash} by {\it smooth real algebraic}.

Schwartz functions are functions that decay, together with all
their derivatives, faster than any polynomial. On $\R^n$ it is the
usual notion of Schwartz function. For precise definitions of
those notions we refer the reader to \cite{AG_Sch}. We will use
the following notations.

\begin{notation}
Let $X$ be a Nash manifold. Denote by $\Sc(X)$ the \Fre space of
Schwartz functions on $X$.

Denote by $\Sc^*(X):=\Sc(X)^*$ the space of Schwartz distributions
on $X$.

For any Nash vector bundle $E$ over $X$ we denote by $\Sc(X,E)$ the
space of Schwartz sections of $E$ and by $\Sc^*(X,E)$ its dual
space.
\end{notation}

\begin{notation}
Let $X$ be a smooth manifold and let $Z \subset X$ be a closed
subset. We denote $\Sc^*_X(Z):= \{\xi \in \Sc^*(X)|\Supp(\xi)
\subset Z\}$.

For a locally closed subset $Y \subset X$ we denote
$\Sc^*_X(Y):=\Sc^*_{X\setminus (\overline{Y} \setminus Y)}(Y)$. In
the same way, for any bundle $E$ on $X$ we define $\Sc^*_X(Y,E)$.
\end{notation}

\begin{remark}
Schwartz distributions have the following two advantages over
general distributions:\\
(i) For a Nash manifold $X$ and an open Nash submanifold $U\subset
X$, we have the following exact sequence
$$0 \to \Sc^*_X(X \setminus U)\to \Sc^*(X) \to \Sc^*(U)\to 0.$$
(ii) Fourier transform defines an isomorphism $\Fou:\Sc^*(\R^n)
\to \Sc^*(\R^n)$.
\end{remark}

\subsubsection{Basic tools}
$ $\\
We present here some basic tools on equivariant distributions that we
will use in  this paper.

\begin{proposition} \label{Strat}
Let a Nash group $G$ act on a Nash manifold $X$.
Let $Z \subset X$ be a closed subset.

Let $Z =
\bigcup_{i=0}^l Z_i$ be a Nash $G$-invariant stratification of
$Z$. Let $\chi$ be a character of $G$. Suppose that for any $k \in
\Z_{\geq 0}$ and $0 \leq i \leq l$ we have $\Sc^*(Z_i,Sym^k(CN_{Z_i}^X))^{G,\chi}=0$. Then
$\Sc^*_X(Z)^{G,\chi}=0$.
\end{proposition}

This proposition immediately follows from Corollary 7.2.6 in
\cite{AGS}.

\begin{theorem}[Frobenius reciprocity] \label{Frob}
Let a Nash group $G$ act transitively on a Nash manifold $Z$. Let
$\varphi:X \to Z$ be a $G$-equivariant Nash map. Let $z\in Z$. Let
$G_z$ be its stabilizer. Let $X_z$ be the fiber of $z$. Let $\chi$
be a character of $G$. Then $\Sc^*(X)^{G,\chi}$ is canonically
isomorphic to $\Sc^*(X_z)^{G_z,\chi \cdot \Delta_G|_{G_z} \cdot
\Delta_{G_z}^{-1}}$ where $\Delta$ denotes the modular character.
\end{theorem}

For proof see \cite{AG_HC}, Theorem 2.3.8.

\subsubsection{Fourier transform}
$ $\\
From now till the end of the paper we fix an additive character
$\kappa$ of $\R$ given by $\kappa(x):=e^{2\pi i x}$.

\begin{notation}
Let $V$ be a vector space over $\R$. Let $B$ be a non-degenerate
bilinear form on $V$. Then $B$ defines Fourier transform with
respect to the self-dual Haar measure on $V$. We denote it by
$\Fou_B: \Sc^*(V) \to \Sc^*(V)$.

For any Nash manifold $M$
we also denote by $\Fou_B:\Sc^*(M \times V) \to \Sc^*(M
\times V)$ the partial Fourier transform.

If there is no ambiguity, we will write $\Fou_V$, and sometimes
just $\Fou$, instead of $\Fou_B$.
\end{notation}

We will use the following trivial observation.

\begin{lemma}
Let $V$ be a finite dimensional vector space over $\rr$. Let a
Nash group $G$ act linearly on $V$. Let $B$ be a $G$-invariant
non-degenerate symmetric bilinear form on $V$. Let $M$ be a Nash
manifold with an action of $G$.  Let $\xi \in \Sc^*(V(\rr) \times
M)$ be a $G$-invariant distribution. Then $\Fou_B(\xi)$ is also
$G$-invariant.
\end{lemma}

\subsection{Gelfand pairs and invariant distributions}
$ $\\
In this section we recall a technique due to Gelfand and Kazhdan
(see \cite{GK}) which allows to deduce statements in
representation theory from statements on invariant distributions.
For more detailed description see \cite{AGS}, section 2.

\begin{definition}
Let $G$ be a reductive group. By an \textbf{admissible
representation of} $G$ we mean an admissible smooth \Fre
representation of $G(\rr).$
\end{definition}

We now introduce three notions of Gelfand pair.

\begin{definition}\label{GPs}
Let $H \subset G$ be a pair of reductive groups.
\begin{itemize}
\item We say that $(G,H)$ satisfy {\bf GP1} if for any irreducible
admissible smooth \Fre representation $(\pi,E)$ of $G$ we have
$$\dim Hom_{H(\rr)}(E,\cc) \leq 1$$

\item We say that $(G,H)$ satisfy {\bf GP2} if for any irreducible
admissible smooth \Fre representation $(\pi,E)$ of $G$ we have
$$\dim Hom_{H(\rr)}(E,\cc) \cdot \dim Hom_{H(\rr)}(\widetilde{E},\cc)\leq
1$$

\item We say that $(G,H)$ satisfy {\bf GP3} if for any irreducible
{\bf unitary} representation $(\pi,\mathcal{H})$ of $G(\rr)$ on a
Hilbert space $\mathcal{H}$ we have
$$\dim Hom_{H(\rr)}(\mathcal{H}^{\infty},\cc) \leq 1.$$
\end{itemize}

\end{definition}
Property GP1 was established by Gelfand and Kazhdan in certain
$p$-adic cases (see \cite{GK}). Property GP2 was introduced in
\cite{Gross} in the $p$-adic setting. Property GP3 was studied
extensively by various authors under the name {\bf generalized
Gelfand pair} both in the real and $p$-adic settings (see
e.g.\cite{vD,Bos-vD}).

We have the following straightforward proposition.

\begin{proposition}
$GP1 \Rightarrow GP2 \Rightarrow GP3.$
\end{proposition}

We will use the following theorem from \cite{AGS} which is a
version of a classical theorem of Gelfand and Kazhdan.

\begin{theorem}\label{DistCrit}
Let $H \subset G$ be reductive groups and let $\tau$ be an
involutive anti-automorphism of $G$ and assume that $\tau(H)=H$.
Suppose $\tau(\xi)=\xi$ for all bi $H(\rr)$-invariant
distributions $\xi$ on $G(\rr)$. Then $(G,H)$ satisfies GP2.
\end{theorem}

In our case GP2 is equivalent to GP1 by the following proposition.

\begin{proposition}
Suppose $H \subset \mathrm{GL}_{n}$ is transpose invariant
subgroup. Then $GP1$ is equivalent to $GP2$ for the pair
$(\mathrm{GL}_{n},H)$.
\end{proposition}

For proof see \cite{AGS}, proposition 2.4.1.

\begin{corollary}
Theorem A implies Theorem B.
\end{corollary}

\subsection{Symmetric pairs}
$ $\\
In this subsection we review some tools developed in \cite{AG_HC}
that enable to prove that, granting certain hypothesis, a
symmetric pair is a Gelfand pair.

\begin{definition}
A \textbf{symmetric pair} is a triple $(G,H,\theta)$ where
$H\subset G$ are reductive groups, and $\theta$ is an involution
of $G$ such that $H = G^{\theta}$. In cases when there is no
ambiguity we will omit ${\theta}$

For a symmetric pair $(G,H,\theta)$ we define an anti-involution
$\sigma :G \to G$ by $\sigma(g):=\theta(g^{-1})$, denote $\g:=Lie
G$, $\h := LieH$, $\gd:=\{a \in \g | \theta(a)=-a\}$. Note that
$H$ acts on $\gd$ by the adjoint action. Denote also
$G^{\sigma}:=\{g \in G| \sigma(g)=g\}$ and define a
\textbf{symmetrization map} $s:G(\rr) \to G^{\sigma}(\rr)$ by
$s(g):=g\sigma(g)$.
\end{definition}
The following lemma is standard:
\begin{lemma}\label{symmetrization_is_open}
The symmetrization map $s:G \to G^{\sigma}$ is submersive and
hence open.
\end{lemma}

\begin{definition}
Let $(G_1,H_1,\theta_1)$ and $(G_2,H_2,\theta_2)$ be symmetric
pairs. We define their \textbf{product} to be the symmetric pair
$(G_1 \times G_2,H_1 \times H_2,\theta_1 \times \theta_2)$.
\end{definition}

\begin{definition}
We call a symmetric pair $(G,H,\theta)$ \textbf{good} if for any
closed $H(\rr) \times H(\rr)$ orbit $O \subset G(\rr)$, we have
$\sigma(O)=O$.
\end{definition}

\begin{definition}
We say that a symmetric pair $(G,H,\theta)$ is a \textbf{GK pair}
if any $H(\rr) \times H(\rr)$ - invariant distribution on $G(\rr)$
is $\sigma$ - invariant.
\end{definition}

\begin{definition}
We define an involution $\theta: GL_{2n} \to GL_{2n}$ by
$\theta(x)=J x^t J^{-1}$ where $J = \begin{pmatrix}
  0_{n} & Id_{n} \\
  -Id_{n} & 0_{n}
\end{pmatrix}$. Note that $(GL_{2n},Sp_{2n},\theta)$ is a symmetric pair.
\end{definition}

Theorem A can be rephrased in the following way:

\begin{theoremm1*}
The pair $(GL_{2n},Sp_{2n})$ defined over $\rr$ is a GK pair.
\end{theoremm1*}

\subsubsection{Descendants of symmetric pairs}
\begin{proposition} \label{PropDescend}
Let  $(G,H,\theta)$ be a symmetric pair. Let $g \in G(\rr)$ such
that $HgH$ is closed. Let $x=s(g)$. Then $x$ is semisimple.
\end{proposition}
For proof see e.g. \cite{AG_HC}, Proposition 7.2.1.
\begin{definition}
In the notations of the previous proposition we will say that the
pair $(G_x,H_x,\theta|_{G_x})$ is a descendant of $(G,H,\theta)$.
Here $G_{x}$ (and similarly for $H$) denotes the stabilizer of $x$
in $G$.
\end{definition}

\subsubsection{Regular symmetric pairs}
\begin{notation}
Let $V$ be an algebraic finite dimensional representation over
$\rr$ of a reductive group $G$. Denote $Q(V):=V/V^G$. Since $G$ is
reductive, there is a canonical embedding $Q(V) \hookrightarrow
V$.
\end{notation}

\begin{notation}
Let $(G,H,\theta)$ be a symmetric pair. We denote by
$\mathcal{N}_{G,H}$ the subset of all the nilpotent elements in
$Q(\gd)$. Denote $R_{G,H}:=Q(\gd) - \mathcal{N}_{G,H}$.
\end{notation}
Our notion of $R_{G,H}$ coincides with the notion $R(\gd)$ used in
\cite{AG_HC}, Notation 2.1.10. This follows from Lemma 7.1.11 in
\cite{AG_HC}.
\begin{definition}
Let $\pi$ be an action of a real reductive group $G$ on a smooth
affine variety $X$.
We say that an algebraic automorphism $\tau$ of $X$ is \textbf{$G$-admissible} if \\
(i) $\pi(G(\rr))$ is of index at most 2 in the group of
automorphisms of $X$
generated by $\pi(G(\rr))$ and $\tau$.\\
(ii) For any closed $G(\rr)$ orbit $O \subset X(\rr)$, we have
$\tau(O)=O$.
\end{definition}

\begin{definition}
Let $(G,H,\theta)$ be a symmetric pair. We call an element $g \in G(\rr)$ \textbf{admissible} if\\
(i) $Ad(g)$ commutes with $\theta$ (or, equivalently, $s(g)\in Z(G)$) and \\
(ii) $Ad(g)|_{\g^{\sigma}}$ is $H$-admissible.
\end{definition}

\begin{definition}
We call a symmetric pair $(G,H,\theta)$ \textbf{regular} if for
any admissible $g \in G(\rr)$ such that every $H(\rr)$-invariant
distribution on $R_{G,H}$ is also $Ad(g)$-invariant,
we have\\
(*) every $H(\rr)$-invariant distribution on $Q(\gd)$ is also
$Ad(g)$-invariant.
\end{definition}

Clearly, the product of regular pairs is regular (see
\cite{AG_HC}, Proposition 7.4.4).

We will deduce Theorem A' (and hence Theorem A) from the following
Theorem:

\begin{theoremmm*}
The pair $(GL_{2n},Sp_{2n})$ defined over $\rr$ is regular.
\end{theoremmm*}

The deduction is based on the following theorem (see \cite{AG_HC},
Theorem 7.4.5.):
\begin{theorem} \label{GoodHerRegGK}
Let $(G,H,\theta)$ be a good symmetric pair such that all its
descendants (including itself) are regular. Then it is a GK pair.
\end{theorem}

\begin{corollary} Theorem C implies Theorem A.
\end{corollary}
\begin{proof}
The pair $(GL_{2n},Sp_{2n})$ is good by Corollary 3.1.3 of
\cite{S}. In \cite{S} it is shown that all the descendance of the
pair $(GL_{2n},Sp_{2n})$ are products of pairs of the form
$(GL_{2m},Sp_{2m})$ and $((GL_{2m})_{\C/\R},(Sp_{2m})_{\C/\R})$,
here $G_{\cc/\rr}$ denotes the restriction of scalars (in
particular $G_{\C/\R}(\R)=G(\C)$). By Corollary 3.3.1. of \cite{S}
the pair $((GL_{2m})_{\C/\R},(Sp_{2m})_{\C/\R})$ is regular. Now
clearly Theorem C implies Theorem A' and hence Theorem A.
\end{proof}

We denote by $O(V)$ the space of regular functions on the
algebraic variety $V$. We will also need the following
Proposition, which must be well known.
\begin{proposition}\label{Submersive}
Let $\pi: \g^{\sigma} \to Spec(O(\g^{\sigma}))^{H}$ be the
projection. Let $x \in \mathcal{N}_{G,H}$ be a smooth point. Then
$\pi$ submersive at $x$.

\end{proposition}

\begin{proof}
Let $\mathcal{J}=\{f \in O(\g^{\sigma})^{H}: f(0)=0\}.$ By Theorem
14 of \cite{KR}, $\mathcal{J}$ is a radical ideal. Using the
Nullstellensatz, this implies that $Ker(d_{x}
\pi)=T_{x}(\mathcal{N}_{G,H}).$ This proves that $\pi$ is
submersive.
\end{proof}

\subsection{Singular support of distributions}
$ $\\
In this subsection we introduce an important invariant $SS(\xi)$
of a distribution $\xi$ and list some of its properties. For more
details see \cite{AG_strong}.

\begin{notation}
Let $X$ be a smooth algebraic variety. Let $\xi \in \Sc^*(X(\R))$.
Let $M_{\xi}$ be the $D_{X}$ submodule of $\Sc^*(X(\R))$ generated
by $\xi$. We denote by $SS(\xi) \subset T^*X$ the singular support
of $M_{\xi}$ (for the definition see \cite{Bor}). We will call it
the {\bf singular support} of $\xi$.
\end{notation}

\begin{remark}
$ $ \\
(i) A similar, but not equivalent notion is sometimes called in
the literature a 'wave front of $\xi$'.
\\
(ii) In some of the literature, singular support of a distribution
is a subset of $X$ not to be confused with our $SS(\xi)$ which is
a subset of $T^*X.$ We use terminology from the theory of $D$
modules where the set $SS(\xi)$ is called both the characteristic
variety and the singular support of the $D$ module $M_{\xi}$.
\end{remark}

\begin{notation}
Let $(V,B)$ be a quadratic space. Let $X$ be a smooth algebraic
variety. Consider $B$ as a map $B:V \to V^*$. Identify $T^*(X
\times V)$ with $T^*X \times V \times V^*$. We define $F_V: T^*(X
\times V) \to T^*(X  \times V)$ by $F_V(\alpha,v,\phi):= (\alpha,
-B^{-1}\phi,Bv)$.
\end{notation}

\begin{definition}
Let $M$ be a smooth algebraic variety and $\omega$ be a symplectic
form on it.
Let $Z\subset M$ be an algebraic subvariety. We call it {\bf $M$-co-isotropic} if one of the following equivalent conditions holds.\\
\begin{enumerate}
\item
The ideal sheaf of regular functions that vanish on $\overline{Z}$ is closed under Poisson bracket. \\
\item
At every smooth point $z \in Z$ we have  $T_zZ \supset
(T_zZ)^{\bot}$. Here, $(T_zZ)^{\bot}$ denotes the orthogonal
complement
to $(T_zZ)$ in $(T_zM)$ with respect to $\omega$. \\
\item
For a generic smooth point $z \in Z$ we have $T_zZ \supset
(T_zZ)^{\bot}$.
\end{enumerate}
If there is no ambiguity, we will call $Z$ a co-isotropic variety.
\end{definition}
Note that every non-empty $M$-co-isotropic variety is of dimension
at least $\frac{1}{2}dimM$.

\begin{notation}
For a smooth algebraic variety $X$ we always consider the standard
symplectic form on $T^*X$. Also, we denote by $p_X:T^*X \to X$ the
standard projection.
\end{notation}

Let $X$ be a smooth algebraic variety. Below is a list of
properties of the Singular support. Proofs can be found in
\cite{AG_strong} section 2.3 and Appendix B.

\begin{property} \label{Supp2SS}
Let $\xi \in \Sc^*(X (\R))$.  Then $\overline{\Supp(\xi)}_{Zar} =
p_X(SS(\xi))(\R)$, where $\overline{\Supp(\xi)}_{Zar} $ denotes
the Zariski  closure of
 $\Supp(\xi)$.
\end{property}
\begin{property} \label{Ginv}
\item Let an algebraic group $G$ act on $X$. Let $\g$ denote the
Lie algebra of $G$. Let $\xi \in \Sc^*(X(\R))^{G(\R)}$. Then
$$SS(\xi) \subset \{(x,\phi) \in T^*X \, | \, \forall \alpha \in
\g \, \phi(\alpha(x)) =0\}.$$
\end{property}
\begin{property} \label{Fou}
Let $(V,B)$ be a quadratic space. Let $Z \subset X \times V$ be a
closed subvariety, invariant with respect to homotheties in $V$.
Suppose that $\Supp(\xi) \subset Z(\R)$. Then $SS(\Fou_V(\xi))
\subset F_V(p_{X \times V}^{-1}(Z))$.
\end{property}
Finally, the following is a corollary of the integrability theorem
(\cite{KKS}, \cite{Mal}, \cite{Gab}):

\begin{property} \label{CorSingSup} Let $X$ be a smooth algebraic
variety. Let $\xi \in \Sc^*(X(\R))$. Then $SS(\xi)$ is
co-isotropic.
\end{property}

\section{Invariant distributions supported on non-distinguished nilpotent orbits in symmetric
pairs} \label{sec2}
For this section we fix a symmetric pair $(G,H,\theta).$
\begin{definition}
We say that a nilpotent element $x \in \g^{\sigma}$ is
distinguished if $$\g_{x} \cap Q(\g^{\sigma}) \subset
\mathcal{N}_{G,H}$$
\end{definition}

\begin{theorem}\label{maintrick}
Let $A \subset \mathcal{N}_{G,H}$
be an $H$ invariant closed subset and assume that all elements of
$A$ are non-distinguished. Let $W=\Sc^{*}_{\g^{\sigma}}(A)^{H}$.
Then $W \cap \Fou(W) =\{0\}.$
\end{theorem}

\begin{remark}
We believe that the methods of \cite{SZ} allow to show the same
result without the assumption of $H$-invariance.
\end{remark}
The proof is based on the following proposition:

\begin{proposition}\label{mainprop}
Let $A \subset \mathcal{N}_{G,H}$
be an $H$ invariant closed subset and assume that all elements of
$A$ are non-distinguished. Denote by $B=\{(\alpha,\beta) \in A
\times A: [\alpha,\beta]=0\} \subset Q(\g^{\sigma}) \times
Q(\g^{\sigma})$. Identify $T^{*}(Q(\g^{\sigma}))$ with
$Q(\g^{\sigma}) \times Q(\g^{\sigma}).$
Then there is no non-empty
$T^{*}(Q(\g^{\sigma}))$-co-isotropic subvariety of $B$.
\end{proposition}

\begin{proof}
Stratify $A$ by finite many orbits
$\mathcal{O}_{1},...,\mathcal{O}_{r}$.
Let $p: A \times A \to A$ be the projection onto the first factor.
By inductive argument it is enough to show that, for any orbit
$\mathcal{O}$, $p^{-1}(\mathcal{O})$ does not include a non empty
co-isotropic subvariety. Consider the set
$$C_{\mathcal{O}}=\{(a,b): a \in \mathcal{O}, b \in Q(\g^{\sigma}),
[a,b]=0\}.$$ Then $dim(C_{\mathcal{O}})=dim(Q(\g^{\sigma})).$
Since $\mathcal{O}$ is not distinguished, $p^{-1}(\mathcal{O})$
 is a closed subvariety of $C_{\mathcal{O}}$ which does not
include any of the irreducible components of $C_{\mathcal{O}}.$
This finishes the proof.
\end{proof}

\begin{proof}[Proof of Theorem \ref{maintrick}]
Let $\xi \in W \cap \Fou(W)$ and let $B$ be as in proposition
\ref{mainprop}. By properties \ref{Supp2SS}, \ref{Ginv}, \ref{Fou}
we conclude that $SS(\xi) \subset B$. But by Property
\ref{CorSingSup} it is co-isotropic and hence by Proposition
\ref{mainprop} it is empty. Thus $\xi=0$.
\end{proof}

\section{Regularity}\label{sec3}

In this section we prove the main result of the paper:
\begin{theoremmm*}\label{main} The pair $(GL_{2n},Sp_{2n})$
defined over $\rr$ is regular.
\end{theoremmm*}

For the rest of this section we let $(G,H)$ to be the symmetric
pair $(GL_{2n}(\rr),Sp_{2n}(\rr)).$

\subsection{$H$ orbits on $\g^{\sigma}$}
\begin{proposition}\label{Hdistinguished}
There exists a unique distinguished $H$-orbit in
$\mathcal{N}_{G,H}(\rr)$. This orbit is open in
$\mathcal{N}_{G,H}(\rr)$ and invariant with respect to any
admissible $g \in G.$
\end{proposition}

For the proof we will use the following Proposition (this is
Proposition 2.1 of \cite{GG}):

\begin{proposition} \label{for good}
Let $F$ be an arbitrary field. For  $x \in GL_{n}(F)$ define
$$\gamma(x)=\begin{pmatrix}
  x & 0 \\
  0 & I_{n}
\end{pmatrix}
$$

Then $\gamma$ induces a bijection between the set of conjugacy
classes in $GL_{n}(F)$ and the set of orbits of $Sp_{2n}(F) \times
Sp_{2n}(F)$ in $GL_{2n}(F)$.
\end{proposition}

\begin{corollary}
Let $d: \mathrm{gl}_{n} \to \g^{\sigma}$ be defined by
$$d(X)=\begin{pmatrix}
  X & 0 \\
  0 & X^{t}
\end{pmatrix}
.$$ Then $d$ induces a bijection between nilpotent conjugacy
classes in $\mathrm{gl}_{n}$ and $H$ orbits in
$\mathcal{N}_{G,H}$.
\end{corollary}
\begin{proof}

Let  $s:GL_{2n} \to GL_{2n}^{\sigma}$ be given by
$s(g)=g\sigma(g).$ Let $W=s(GL_{2n}(\rr)).$ By Proposition
\ref{for good}, the map $s \circ \gamma $ induces a bijection
between conjugacy classes in $GL_{n}(\rr)$ and $H$ orbits on $W.$

Let $e:\mathcal{N} \to GL_{n}$ be given by $e(X)=1+X$ where
$\mathcal{N}$ is the cone of nilpotent elements in
$\mathrm{gl}_{n}.$ Let
 $\ell: W \to \g^{\sigma}$ given by $\ell(w)=w-1$.

Then, it is easy to see that the map $d|_{\mathcal{N}}:\mathcal{N}
\to \mathcal{N}_{G,H}$ coincides with the composition $\ell \circ
s \circ \gamma \circ e.$

To finish the proof of the Proposition it is enough to show that
$\ell(W)$ contains all nilpotent elements. Indeed, by lemma
\ref{symmetrization_is_open} the set $W=s(GL_{2n}(\rr))$ is open
and thus $\ell(W)$ is open and hence contains all nilpotent
elements.
\end{proof}

We are now ready to prove the proposition.
\begin{proof}[Proof of Proposition \ref{Hdistinguished}]
It is easy to see that if $X$ is non regular nilpotent then $d(X)$
is not distinguished. Also, a simple verification shows that if
$X=J_{n}$ is a standard Jordan block then $d(J_{n})$ is
distinguished. Thus we only need to show that $C=Ad(H) d(J_{n})$
is open in $\mathcal{N}_{G,H}$. For this we will show that $C$ is
dense in $\mathcal{N}_{G,H}$. Indeed, $\bar{C} \supset d(\overline
{Ad(GL_{n})J_{n}})=d(\mathcal{N})$, where $\mathcal{N}$ is the set
of nilpotent elements in $gl_{n}.$ But $C$ is $Ad(H)$-invariant
and this implies that $\bar{C}=\mathcal{N}_{G,H}$
\end{proof}

\subsection{Proof of Theorem C}
$ $\\
The theorem follows from Theorem \ref{maintrick} and the next
 Proposition:

\begin{proposition}
Let $g \in G$ be an admissible element.  Let $A$ be the union of
all non-distinguished elements. Note that $A$ is closed. Let $\xi$
be any $H$ invariant distribution on $\g^{\sigma}$ which is
anti-invariant with respect to $Ad(g)$. Then $Supp(\xi) \subset
A$.
\end{proposition}

\begin{proof}
Let $O_{0} \subset \mathcal{N}_{G,H}$
be the distinguished orbit. Let $\widetilde{H}=\langle Ad(H),Ad(g)
\rangle$ be the group of automorphisms of $\g^{\sigma}$ generated
by the adjoint action of $H$ and $g$. Let $\chi$ be the character
of $\widetilde{H}$ defined by $\chi(\widetilde{H}-H)=-1.$ We need
to show
$$\Sc^{*}_{Q(g^{\sigma})}(O_{0})^{\widetilde{H},\chi}=0$$
By Proposition \ref{Strat} it is enough to show
$$\Sc^{*}(O_{0},Sym^{k}(CN_{O_{0}}^{Q(\g^{\sigma})}))^{\widetilde{H},\chi}=0$$
Notice that $\widetilde{H}$ acts trivially on
$Spec(O(\g^{\sigma}))^{H}$. Hence, by Proposition \ref{Submersive}
the bundle $N_{O_{0}}^{Q(\g^{\sigma})}$ is trivial as a
$\widetilde{H}$ bundle. This completes the proof.
\end{proof}

\end{document}